\documentclass[a4paper]{amsart}
\usepackage[
left=35truemm,
right=35truemm
]{geometry}

\usepackage{amsfonts, amssymb, amsmath, verbatim, amsthm, mathrsfs, amscd, enumerate, ascmac, fancyhdr, caption, framed, subfigure}

\usepackage{doi}

\usepackage{bbm}

\usepackage{hyperref}
\usepackage{bookmark}
\usepackage{physics}

\numberwithin{equation}{section}

\usepackage[ 
]
{hyperref} 
\usepackage{xcolor}

\usepackage{tikz}
\numberwithin{figure}{section}
\usetikzlibrary{positioning,intersections, calc, arrows,decorations.markings, angles}

\usepackage[nolists]{endfloat}

\def\d{\mathrm{d}}

\newcommand{\R}{\mathbb{R}}
\newcommand{\Z}{\mathbb{Z}}
\newcommand{\N}{\mathbb{N}}
\newcommand{\C}{\mathbb{C}}
\newcommand{\T}{\mathbb{T}}
\newcommand{\lam}{\lambda}
\newcommand{\supp}{\mathrm{supp}\thinspace}
\newcommand{\F}{\mathcal{F}}
\newcommand{\wdh}[1]{\widehat{#1}}

\newcommand{\ls}[3]{L^{#1}_{#2}(#3)}

\newcommand{\sq}[1]{\square_{#1}}

\usepackage{mathtools}
\DeclarePairedDelimiter\ang{\langle}{\rangle}

\usepackage{xcolor} 

\theoremstyle{plain}
\newtheorem{thm}{Theorem}[section]
\newtheorem*{thm*}{Theorem}
\newtheorem{prop}[thm]{Proposition}
\newtheorem{lem}[thm]{Lemma}
\newtheorem{cor}[thm]{Corollary}

\newtheorem{rem}[thm]{Remark}
\theoremstyle{definition}
\newtheorem{dfn}[thm]{Definition}

\title[Reverse square function estimates]{Reverse square function estimates for degenerate curves and their applications}

\author[Bulj]{Aleksandar Bulj}
\address[Aleksandar Bulj]{Department of Mathematics, Faculty of Science, University of Zagreb, Bijeni\v{c}ka cesta 30, 100000 Zagreb, Croatia}
\email{aleksandar.bulj@math.hr}

\author[Inami]{Kotaro Inami}
\address[Kotaro Inami]{Current address: Department of Mathematics, Southern University of Science and Technology, Shenzhen, China\newline
Previous address: Graduate School of Mathematics, Nagoya University, Furocho, Chikusaku, Nagoya, Japan}
\email{inami@sustech.edu.cn}

\author[Shiraki]{Shobu Shiraki}
\address[Shobu Shiraki]{Department of Mathematics, Faculty of Science, University of Zagreb, Bijeni\v{c}ka cesta 30, 100000 Zagreb, Croatia}
\email{shobu.shiraki@math.hr}

\begin{document}

\date{\today}


\begin{abstract}
Building on the classical work of C\'{o}rdoba–Fefferman and the recent work of Schippa, we establish $L^4$ reverse square function estimates for functions whose Fourier support is contained in a $\delta$-neighborhood of the curve $\{(\xi,|\xi|^a): |\xi|\leq 1\}$ in $\R^2$, for all exponents $a\in(0,\infty)\setminus\{1\}$. As applications, we derive sharp $L^4$ Strichartz estimates on the one-dimensional torus for fractional Schrödinger equations and establish new local smoothing estimates in modulation spaces. In the latter application, orthogonal Strichartz-type estimates also play a crucial role.
\end{abstract}

\maketitle

\section{Introduction}
\subsection{Main result}

A simple but powerful way to analyze functions whose Fourier support is contained in a $\delta$-neighbourhood of a non-degenerate curve in $\R^2$, such as a parabola, is to cover this neighbourhood by rectangles of size $\delta^{\frac12}\times\delta$ and decompose the Fourier support accordingly. Here, a curve is called non-degenerate if its curvature does not vanish anywhere, and degenerate otherwise. This is precisely the geometric setup underlying the classical C\'ordoba–Fefferman reverse square function estimate for the parabola \cite{Cordoba79, Cordoba82,Fefferman73}.

Very recently, Schippa considered degenerate curves of the form $\{(\xi,|\xi|^k):|\xi|\le1\}$ for integers $k\ge3$ \cite{schippa2024generalized} and $k=3$ \cite{Schippa2025quasi}, and obtained an analogous reverse square function estimate. The main difficulty in this degenerate setting is that the curvature is no longer uniform. In fact, a rectangle of size $\delta^{\frac1k}\times\delta$ captures the curve well near the origin, but becomes too large away from the origin, where the curvature increases. This phenomenon can be seen directly from a Taylor expansion of the curve.


To address this issue, Schippa introduced two types of localization: a location-dependent decomposition in \cite{schippa2024generalized} and a curved decomposition in \cite{Schippa2025quasi}. In the location-dependent decomposition, one takes a horizontal rectangle of size comparable to $\delta^{1/k}\times\delta$ around the origin, while elsewhere one take rectangles centered at $(\xi,\xi^k)$ whose tangential length is $\delta^{\frac12}|\xi|^{-\frac{k-2}{2}}$ and whose normal width is $\delta$. 
The curved localization, on the other hand, is defined by uniformly partitioning the domain of the first variable at scale $\delta^{1/k}$ . This is the approach we adopt in the present manuscript since it is uniform in scale and well suited to our applications. A more precise definition will be given shortly.

For $a>0$, denote the $\delta$-neighbourhood of the curve $\Gamma_a:=\{(\xi,|\xi|^a)\in \R^2:|\xi|\le1\}$ by
$\Gamma_a(\delta)$, i.e.
\[
\Gamma_a(\delta)
:= 
\{(\xi_1,|\xi_1|^a+\tau): |\xi_1|\leq1,\ |\tau|\leq \delta\}\subset \R^2.
\]
We denote by $\Theta_k\subset \R^2$ the collection of frequency pieces arising from either of the two localizations described above.

\begin{thm}[\cite{Cordoba79, Cordoba82, Fefferman73, schippa2024generalized, Schippa2025quasi}]
Let $k\geq2$ be an integer. For $F\in\mathcal{S}(\R^2)$ with $\supp \widehat{F}\subset \Gamma_k(\delta)$, we have 
\[
\|F\|_{L^4(\R^2)}
\lesssim
\Big\|\Big( \sum_{\theta\in\Theta_{k}} |F_{\theta} |^2 \Big)^\frac12\Big\|_{L^4(\R^2)}.
\]
Here, $F_\theta$ denotes $F$ with its Fourier support restricted to $\theta\in\Theta_k$.
\end{thm}

The purpose of the present manuscript is to establish an analogous reverse square function estimate for more general curves of the form $(\xi,|\xi|^a)$, not only for integer exponents but for all $a\in(0,\infty)\setminus\{1\}$. We also investigate the sharp $\delta$-loss in the reverse square function estimate arising from different decomposition scales, parametrized by $b>0$. Since we investigate the reverse square function estimates with decompositions of a uniform length, our estimates do not follow from the constant curvature cases via a rescaling argument. 

Let $\Omega=\Omega_b(\delta)$ be a $\ge 2\delta^{1/b}$-separated subset of $[-1,1]$. 
For a Schwartz function $F:\R^2\to \C$ and $\omega\in\Omega_b(\delta)$, we define the function $F_{\omega}$ by
\[\widehat{F_{\omega}}(\xi):= \widehat{F}(\xi)1_{[\omega-\delta^{1/b},\omega+\delta^{1/b}]}(\xi_1).\]

Our main result is the following.
\begin{thm}\label{t:main}
Let $a\in(0,\infty)\setminus\{1\}$ and $b>0$. For $\delta>0$ and a function $F:\R^2\to\C$ with $\supp \widehat{F}\subset \Gamma_a(\delta)$, let $R_{a,b}(\delta)$ denote the smallest constant for which the inequality
\begin{equation}\label{e:main}
\|F\|_{L^4(\R^2)}
\le R_{a,b}(\delta)
\Bigl\| \Bigl(\sum_{\omega\in \Omega_b(\delta)} |F_{\omega}|^2\Bigr)^{\frac12} \Bigr\|_{L^4(\R^2)}
\end{equation}
holds for every $\ge 2\delta^{1/b}$-separated set $\Omega_b(\delta)$. Then
\[
R_{a,b}(\delta)\lesssim (\delta^{-1})^{\frac{\rho(a,b)}{4}},
\]
where
\begin{equation}\label{c:main}
\rho(a,b)
:=
\max\Bigl\{
0,\,
\frac1b-\frac1a,\,
\frac1b-\frac12
\Bigr\}.
\end{equation}
Moreover, this estimate is sharp up to the implied multiplicative constant.
\end{thm}

It is worth noting that although we decompose the function $F$ into pieces whose Fourier supports lie in vertical strips, the additional assumption that $F$ is Fourier supported in $\Gamma_a(\delta)$ introduces nontrivial geometric considerations. 

We now introduce further notation. 
For $a>0$ fixed and $\omega\in \Omega_b(\delta)$, define
\[\theta_\omega :=\bigl\{(\xi_1,\xi_2)\in \Gamma_a(\delta):\; |\xi_1-\omega|\leq \delta^{1/b}\bigr\}.\]
Observe that if $F$ is Fourier supported in $\Gamma_a(\delta)$, then $F_\omega$ is Fourier supported in $\theta_\omega$.

By Taylor's theorem, for $\xi_1\in [\omega - \delta^{1/b}, \omega + \delta^{1/b}]\subset(0,\infty)$,  we have
\[\xi_1^a-(\omega^a+a\omega^{a-1}(\xi_1-\omega))=  \frac{a(a-1)}{2}\omega^{a - 2}(\xi_1 - \omega)^2+ O_a((\xi_1 - \omega)^3).\]
When $a\neq 1$, the right-hand side is $\lesssim_a \delta^{2/b}$. 
Consequently, when $b>2$, the set $\theta_\omega$ is not contained in any 
$O(\delta^{1/b})\times O(\delta)$ rectangle centered at $\omega$. 
In particular, the sets $\theta_\omega$ are genuinely curved. 
As a result, due to Fefferman's result \cite{Fefferman71}, one cannot define the decomposition of $F$ simply by cutting its Fourier transform with the indicator function $1_{\theta_\omega}$ and we cannot count the overlap of Minkowski sums of the form
$\theta_{\omega}+\theta_{\omega'}$ using properties of rectangles. 
For the special case $a=b=3$, Schippa \cite{Schippa2025quasi} exploited the algebraic structure of the curve $\Gamma_3$ to estimate these overlaps.
The main new ingredient in the proof of Theorem~\ref{t:main} is a counting argument that controls these overlaps using only the first two derivatives of the curve. 
This approach allows us to extend the result to all $a>0$ with $a\neq 1$.

It is worth noting that the constant in \eqref{e:main} becomes 
\[
\delta^{-\frac{\rho(a,b)}{4}}
=
\begin{cases}
    1 \qquad &  b\geq a\\
    \delta^{\frac14(\frac1a-\frac 1b)} \qquad & b<a.
\end{cases}
\]
when $a\geq  2$, and 
\[
\delta^{-\frac{\rho(a,b)}{4}}
=
\begin{cases}
    1 \qquad &  b\geq 2\\
    \delta^{\frac14(\frac12-\frac 1b)} \qquad & b<2
\end{cases}
\]
when $a\in(0,2)\setminus\{1\}$.

By Minkowski's inequality, Theorem~\ref{t:main} immediately yields the $L^4$ decoupling estimate for a function $F$ with Fourier support inside $\Gamma_a(\delta)$. 
\begin{equation}\label{e:decoupling form}
\|F\|_{L^4(\R^2)}
\lesssim
\delta^{-\frac{\rho(a,b)}{4}}
 \Big(\sum_{\omega\in\Omega_b(\delta)} 
\|F_{\omega} \|_{L^4(\R^2)}^2\Big)^\frac12.
\end{equation}
This is the simplest form of $\ell^2$ decoupling inequality. It is known that $\ell^2$ decoupling inequality famously implies the Vinogradov mean value theorem \cite{BDG_vinogradov}, which counts integer solutions to the Vinogradov system.

Motivated by this connection, Gressman et al. \cite{GGPRY21} investigated whether a converse phenomenon holds, namely whether bounds on the number of solutions to such arithmetic systems can imply the decoupling or the reverse square function estimates. They showed that this is indeed the case for extension operators associated with certain non-degenerate curves, such as the moment curve, in some specific regimes. Subsequently, Biggs--Brandes--Hughes \cite{BiggsBrandesHughes22} extended these results to certain degenerate curves in $\mathbb{R}^2$, such as $\xi \mapsto |\xi|^k$ for integers $k \ge 3$ by counting the integer solutions to the system up to the small perturbation 
\begin{align}\label{system a}
     \xi_1+\xi_2 &=\xi_3+\xi_4\nonumber \\
     \xi_1^a+\xi_2^a &=\xi_3^a+\xi_4^a
\end{align}
for $a=k$. They crucially exploit some algebraic structures. As mentioned earlier, Schippa \cite{Schippa2025quasi, schippa2024generalized} has recently treated the reverse square function estimate in the case $a=k$ as in Theorem~\ref{t:main}, where arithmetic computations remain available. It is noteworthy that this system was also treated in the context of Waring's problem (see \cite{Hua_additive}).

It is also worth mentioning at this point that, although the $\ell^2$-decoupling inequality associated with, at least, $a>1$ can be deduced from the parabolic case $a=2$ via so-called induction-on-scales argument (for more details, for instance, see \cite[Proposition~12.20]{Demeter}), it is, to the best of our knowledge, unclear whether an analogous reduction is available at the level of reverse square function estimates. For this reason, we adopt a more direct approach in the proof of our theorem.

In our case, of course, arithmetic advantages that arise from the integer exponent $a$ are no longer available and we therefore employ an analytic approach rather than relying on geometric or arithmetic considerations. 
In order to estimate the number of solutions to the perturbed system \eqref{system a}, 
we reinterpret the discrete counting problem as the problem of estimating the volume of a corresponding $\delta^{1-\frac ab}$-neighborhood of a continuous tube-like region, which can be obtained relatively easy by classical van der Corput's lemma.

The reverse square function estimate admits a broad range of applications, including the corresponding decoupling form. Below, we present two novel applications yielding new results for one-dimensional Strichartz estimates on the torus and for local smoothing estimates in modulation spaces.

\subsection{Strichartz estimates on the torus}

The first application concerns one-dimensional Strichartz estimates in the periodic setting. Our goal is to determine the minimal Sobolev regularity $s\ge 0$ for which
\begin{equation}\label{e:Strichartz periodic}
\bigl\|e^{it(-\partial_x^2)^{\frac a2}} f\bigr\|_{L^4_{t,x}([-1,1]\times\T)}
\lesssim
\|f\|_{H^s(\T)}.
\end{equation}

This problem was first proposed in the case $a=2$ by Fefferman and soon later solved by Zygmund \cite{Zygmund74}, who showed that \eqref{e:Strichartz periodic} holds if and only if $s\geq0$. 
Bourgain \cite{Bourgain_GAFA} further developed this type of estimate and applied it to the study of the local and global well-posedness for the periodic NLS. 
Motivated by models arising in quantum physics \cite{Laskin00,Laskin02} and biophysics \cite{KLS13}, the fractional case $a\in(1,2)$ has since been well understood as well. It is known that \eqref{e:Strichartz periodic} holds if and only if $s\geq \frac14(1-\frac a2)$.  
This result was first established up to the endpoint by Demirbas--Erdo\u{g}an--Tzirakis \cite{DET13}, and the endpoint case itself was proved very recently by Megretski--Skouloudis \cite{megretski_skouloudis2025}. The local and global well-posedness of the corresponding fractional cubic Schr\"{o}dinger equation was established in \cite{CHKL2015} and \cite{megretski_skouloudis2025} respectively. 

As a direct application of our reverse square function estimate (in fact, its decoupling form \eqref{e:decoupling form} is sufficient for this purpose) we cover all the case $a\in (0,\infty)\setminus\{1\}$. 
The behavior naturally splits into two distinct regimes, namely the case $a\geq 2$ and the case $a<2$.
The crucial choice of scale\footnotemark in Theorem~\ref{t:main} is $b=a$.

\begin{cor}
\label{c:periodic_Strichartz}
We have the following. 
\begin{enumerate}[(i)]
    \item\label{item:a>2 Strichartz} For $a\in[2,\infty)$, 
    \[
    \bigl\|e^{ it (-\partial_x^2)^\frac a2} f(x)\bigr\|_{L_{t,x}^4([-1,1]\times \T)} 
    \lesssim
    \|f\|_{L^2(\T)}
    \]
    holds. 
    This result is sharp, in the sense that the estimate fails if the $L^2$ norm is replaced by $\|f\|_{H^s}$ for any $s<0$. 
    
    \item\label{item:a<2 Strichartz} For $a\in(0,2)\setminus\{1\}$,
        \[
    \bigl\|e^{ it (-\partial_x^2)^\frac a2} f(x)\bigr\|_{L_{t,x}^4([-1,1]\times \T)} 
    \lesssim
    \|f\|_{H^s(\T)}
    \]
    holds if and only if $s\geq \frac14(1-\frac a2)$.
\end{enumerate}
\end{cor}

\footnotetext{In the second regime $a\in(0,2)\setminus\{1\}$, the choice $b=2$, combined with Bernstein’s inequality, also leads to the same conclusion.}

\begin{rem}
    By using this Strichartz estimate with $a > 2$ and following the same strategy in \cite{hayashi2025uniqueness}, we can also show the uniqueness of weak $L^2$ solutions to the following the fractional logarithmic nonlinear Schr\"{o}dinger equations: 
    \begin{align*}
    \begin{cases}
        iu_{t}(x,t) - (-\partial_x^2)^{\tfrac a2}u(x,t) = \lam u\log(|u|^2)\quad (x,t) \in \T\times\R, \\
        u(x,0) = f(x)
    \end{cases}
\end{align*}
where $\lam\in \R \backslash \{0\}$. 
\end{rem}

\subsection{Local smoothing estimates in modulation spaces}
Modulation spaces were first introduced in \cite{feichtinger1983modulation}. Although they were originally used in time-frequency analysis, it turned out that they were also useful for nonlinear partial differential equations; for instance, see \cite{wang_henryk2007}, \cite{guo2017almost}, \cite{oh_wang2020}, \cite{Schippa2022}, \cite{klaus2023}, and \cite{schippa2023infinite}. Modulation spaces are designed to treat the position coordinate and the frequency coordinate in a symmetric way. This is a main difference from Sobolev spaces and Besov spaces. Let us recall the definition of modulation spaces. 
\begin{dfn}[Modulation spaces]
    Let $\varphi\in C^{\infty}_{0}(\R^d)$ satisfy
    \begin{align*}
        \begin{cases}
            \mathrm{supp}\thinspace\varphi \subset [-1,1]^d\\
            {\displaystyle\sum_{k\in\Z^{d}}}\varphi(\xi-k) = 1$ for all $\xi\in\R^d
        \end{cases}
    \end{align*}
    and define
    \begin{equation*}
    \widehat{\square_{k}f}(\xi) := \varphi(\xi - k)\widehat{f}(\xi), \quad \xi\in\R^d. 
    \end{equation*}
    The (quasi) norm of modulation spaces is defined for any $f\in\mathcal{S}'(\R^d)$, $s\in\R$, and $p,q\in (0,\infty]$ as follows:
    \begin{equation*}
        \|f\|_{M^{s}_{p,q}(\R^d)} := \Bigl(\|\ev{k}^{sq}\|\square_{k}f(x)\|^q_{L^{p}_{x}(\R^d)}\Bigr)^{1/q}
    \end{equation*}
    where $\ev{k} := (1 + |k|^{2})^{\frac{1}{2}}$. For $s\in\R$, $p,q\in(0,\infty]$, modulation spaces are defined as  
    \begin{align*}
        M^{s}_{p,q}(\R^{d}) &:= \{f\in\mathcal{S}'(\R^{d})\thinspace;\thinspace \|f\|_{M^{s}_{p,q}} < \infty\}. \\
    \end{align*}
\end{dfn}
Note that modulation spaces are quasi-Banach spaces and do not depend on the choice of window function $\varphi$. 
Moreover, the modulation spaces enjoy embedding property in each indices: 
    Let $s_1,s_2\in \R$, and $p_1,p_2, q_1, _2\in [1,\infty]$ satisfy 
    \[
    s_1\geq s_2, \quad p_1\leq p_2,\quad q_1\leq q_2.
    \]
    Then, it holds that 
    \[
    M_{p_1,q_1}^{s_1} (\R^d)
    \hookrightarrow 
    M_{p_2,q_2}^{s_2} (\R^d).
    \]

One advantage of modulation space in nonlinear PDE is they can tract slowly decaying initial data. In other words, for $s\geq 0, \; p>2$, and $q\ge 1$, the modulation space $M^{s}_{p,q}(\R)$ is not contained in $H^{s}(\R)$.\footnote{As briefly mentioned in \cite{schippa2023infinite}, we can easily construct an example for this as follows: Let $B(0,n)$ denote the interval centered at the origin with radius $n$ and set $\chi_{1}\in C^{\infty}_{0}(\R)$ with $\supp \chi_{1}\subset B(0,1)$. Also, set $g_{n} := \chi_{1}* \chi_{B_{(0,n)}}$. Obviously, $g_{n}$ are supported on $B(0,n+1)$ and their Fourier transform are rapidly decaying functions. We also normalize these functions as 
$f_{n} := \frac{g_{n}}{\|g_{n}\|_{L^{p}}}$. This sequence violates the inclusion relation $M^{s}_{p,q} \subset H^{s}$. } This means that we can manage initial data with the mass $\|\cdot\|_{L^2}$ and the energy norm $\|\cdot\|_{H^1}$ are not necessarily finite. 


For the linear solution to the standard Schr\"{o}dinger equation, corresponding to the case $a=2$, Schippa \cite{Schippa2022} recently obtained some local smoothing estimates in modulation spaces, in the form of 
\begin{equation}\label{e:local smoothing modulation a=2}
    \|e^{-it\partial_x^2}f\|_{L_{t,x}^r([-1,1]\times\R)}
    \lesssim
    \|f\|_{M_{p,q}^s(\R)}.
\end{equation}
The essential results he obtained are 
$(p,q,r,s)=(6,2,6,\varepsilon),\, (4,2,4,0)$ for arbitrary small $\varepsilon>0$.
Although the argument applies in general dimensions, we restrict our attention to the one-dimensional case, which suffices for the purposes of the present paper. 
The main ingredient in \cite{Schippa2022} is the $\ell^2$-decoupling inequality.
Shortly thereafter, the second author further developed this approach using reverse square function estimates, in particular of C\'{o}rdoba--Fefferman type, and obtained $(p,q,r,s)=(2,4-\varepsilon', 4, \varepsilon)$ for arbitrary small $\varepsilon,\varepsilon'>0$.
Lu \cite{Lu2023} and Chen et al. \cite{chen2024smoothing} showed similar estimates in different settings.


It is also noteworthy that in Rozendaal--Schippa \cite{Rozendaal_Schippa2023} and Hassell et al. \cite{hassell2023function}, function spaces adapting to the decoupling estimate for the cone are proposed. In these function spaces, one can show the local smoothing type estimate by using the decoupling inequality. Although modulation spaces were made in the context of the time--frequency analysis, they naturally have the similar properties with such function spaces for the decoupling.



In the same spirit as \cite{inami2025}, we also obtain new local smoothing estimates in modulation
spaces by using our reverse square function estimates

\begin{cor}\label{c:local_smoothing_modulation}
    The following hold. 
    \begin{enumerate}[(i)]
    \item 
For 
\[
\begin{cases}
a\in(0,6)\setminus\{1\}\ \text{and}\  
s\in[\frac{2-a}{8},\frac14]\\
a \geq 6$\  \text{and}\  $s\in (-\frac12,\frac14), 
\end{cases}
\]
we have 
\begin{equation}\label{eq:localsmoothing_BLN}
\bigl\|e^{it(-\partial^2_x)^\frac a2} f\bigr\|_{L_{t,x}^4([-1,1]\times \R)}
\lesssim_\varepsilon
\|f\|_{M_{2,q}^s(\R)}
\end{equation}
with $q$ such that 
\[
\frac{1}{q}=\frac{a+1}{4a}-\frac{s}{a}.
\]
        \item For $a>2$, we have 
        \begin{equation}
            \bigl\|e^{it(-\partial_x^2)^{\frac a2}} f\bigr\|_{L_{t,x}^4([-1,1]\times \R)}
            \lesssim
            \|f\|_{M_{2,4}^0(\R)}.\label{eq:localsmoothing_Bilinear}
        \end{equation}

    \end{enumerate}
\end{cor}
The proof of \eqref{eq:localsmoothing_BLN} and \eqref{eq:localsmoothing_Bilinear} in Corollary~\ref{c:local_smoothing_modulation} consist of two crucial steps. The first is the reverse square function estimate (Theorem~\ref{t:main}). It is worth emphasizing that this step cannot be replaced by the decoupling estimate \eqref{e:decoupling form}. The second ingredient is the Strichartz estimates for orthogonal systems and their adapted version tailored to the local smoothing estimates in modulation spaces. The details will be discussed in the forthcoming section.

To conclude this section, we briefly mention an improvement of another previously known result. 

In \cite{Lu2023}, among other results including extensions to the more general setting of $\alpha$-modulation space, Lu proved that for $a>2$ and arbitrary small $\varepsilon>0$
\[
    \bigl\|e^{it(-\partial^2_x)^{\frac a2}}f(x)\bigr\|_{L^{4}_{t,x}([-1,1] \times \R)}
    \lesssim_{\varepsilon} 
    \|f\|_{M^{s+\varepsilon}_{4,2}(\R)}
\]
holds with $s=\frac{a-2}{8}$. The following result shows that the required regularity can in fact be taken arbitrarily close to~0 in this special case. 
In contrast to Corollary~\ref{c:local_smoothing_modulation}, we apply \eqref{e:decoupling form} and the local smoothing estimate in Sobolev spaces by \cite{rogers_seeger2010}.

\begin{cor}
For $a\geq 2$ and arbitrary small $\varepsilon>0$, we have
    \begin{equation}
            \bigl\|e^{it(-\partial_{x}^2)^{\frac a2}}f(x)\bigr\|_{L^{4}_{t,x}([-1,1] \times \R)}
            \lesssim_{\varepsilon} 
            \|f\|_{M^{\varepsilon}_{4,2}(\R)}. 
    \end{equation}\label{eq:localsmoothing_lu}
    \label{c:local_smoothing_modulation_lu}
\end{cor}

In \cite{Lu2023}, a local well-posedness result for the fourth order cubic nonlinear Schr\"{o}dinger equation was proved by using the above local smoothing estimate. As a consequence of Corollary \ref{c:local_smoothing_modulation_lu}, we can also refine this local well posedness result (see Subsection \ref{subsection wellposedness}).

\subsection{Orthogonal Strichartz estimates}
The classical Strichartz estimate (for a single initial data) asserts that for $a\in \R\setminus\{0,1\}$, $(q,r)$ satisfying $\frac1q\leq \frac12(\frac12-\frac1r)$
and $s=\frac12-\frac1r-\frac aq$ it holds that 
\[
\bigl\|e^{it(-\partial_x^2)^{\frac a2}} h\bigr\|_{L_t^qL_x^r(\R^2)}
\lesssim
\|h\|_{\dot{H}^s(\R)}.
\]
For further details, see for instance \cite{COX10}.

The orthogonal Strichartz estimate is related to the topic of the reverse square function estimates as mentioned in the previous section. It is formulated as 
\begin{equation}\label{e:orthogonal general}
    \Big\| \Big( \sum_k |e^{it(-\partial_x^2)^{\frac a2}} h_k |^2\Big)^\frac12 \Big\|_{L_t^qL_x^r(\R\times\R)}
    \lesssim
    \Big( \sum_k \|h_k\|_{\dot{H}^s} ^{2\beta}\Big)^{\frac{1}{2\beta}}
\end{equation}
for all orthogonal functions $(h_k)_k$ in $\dot{H}^{s}(\R)$. 
The line of research surrounding \eqref{e:orthogonal general} was initiated by Frank et al. \cite{FLLS} in the non-degenerate case $a=2$, and subsequently developed in a number of works. The interested reader may consult \cite{BHLNS,FS_AJM,FS_Survey} and \cite{BNS_selecta,BKS_TLMS,Lieb83} for the boundary cases.
Numerous variations have been studied in recent years, including Strichartz estimates on the torus by Nakamura \cite{Nakamura}, but we do not attempt to survey other examples here.

In their study of the fractional Schr\"odinger propagator, including higher-dimensional cases, Bez--Lee--Nakamura \cite{BLN_Forum} obtained several results. In what follows, we state only the one-dimensional spatial case, which is directly relevant to our analysis.

\begin{thm}[\cite{BLN_Forum}]\label{t:BLN}
For $a\in (0,\infty)\setminus\{1\}$ and $p,r\geq 2$, suppose $a,p,q$ satisfy 
\[
\frac1q\leq \frac12(\frac12-\frac1r),
\qquad 
\frac1r+\frac aq <1.
\]
Then, the estimate \eqref{e:orthogonal general} holds
for all orthogonal functions $(h_k)_k$ in $\dot{H}^s$ whenever 
\[
s
=
\frac12-\frac1r-\frac aq,
\qquad 
\beta
=
\frac{qr}{2(q+r)}
\]
\end{thm}

Let $h_k=\square_kf$ and replace the measure $\d t$ by $w(t)\d t$, where $w(t)=(1+|t|)^{-100}$. This modification is motivated by applications to the local smoothing estimates in the modulation spaces discussed in the previous section, after applying our reverse square function estimates. In this setting, we obtain the following estimate, which is analogous to Theorem~\ref{t:BLN}. 
\begin{prop}\label{t:orthogonal new}
    For $a>2$, we have 
    \begin{equation}\label{e:orthogonal new}
    \Big\| \Big( \sum_k |e^{it(-\partial_x^2)^{\frac a2}} \square_kf |^2 \Big)^\frac12 \Big\|_{L^4_{t,x}(w\times\R)}
    \lesssim
    \Big( \sum_k \| \square_k f\|_{L^2}^4\Big)^\frac14.
    \end{equation}
\end{prop}

This estimate lies outside the scope of Theorem~\ref{t:BLN} since it particularly implies 
\[
    \Bigl\| \Bigl( \sum_k |e^{it(-\partial_x^2)^{\frac a2}} \sq{k}f |^2 \Bigr)^\frac12 \Bigr\|_{L^4_{t,x}(w\times\R)}
    \lesssim
    \Bigl( \sum_k \| \sq{k} f\|_{L^2}^{2\beta}\Bigr)^\frac{1}{2\beta}
\]
for $a> 2$ and $\beta=\frac{2a}{a+1}<2$.

A standard approach to establishing estimates for orthogonal systems such as \eqref{e:orthogonal general} proceeds as follows.
First, one reformulates \eqref{e:orthogonal general} into an equivalent dual formulation involving appropriate Schatten class operators, equipped with the corresponding Schatten norms $\|\cdot\|_{\mathcal{C}^{\beta'}}$, where $\beta\in[1,\infty]$ is the same exponent appearing on the right-hand side of \eqref{e:orthogonal general}. One then establishes bounds in the special cases $\beta=1,2$ (and $\beta=\infty$, if necessary), where nice representations of the Schatten norms are available. In the context of Strichartz estimates, this is the stage at which oscillatory integral bounds, such as dispersive estimates, may be applied. Additional localizations are typically arranged at each step, and how these are implemented is one of the more delicate aspects of the argument. Finally, the desired bound for general $\beta$ is obtained by an appropriate interpolation argument between the cases $\beta=1,2,\infty$. Further details may be found in the references cited above.

By contrast, we employ a more direct approach to prove Proposition~\ref{t:orthogonal new}, which avoids an initial duality argument and instead exploits the underlying $L^4$ structure. In particular, a notable aspect of the proof is the use of the classical bilinear identity 

\begin{equation}\label{e:bilinear}
    \int_{\R} 
    |Ef(x) Eg(x)|^2
    \,\d x
    =
    (2\pi)^2
    \int_{\R^2} \frac{|\widehat{f}(\xi_1)|^2 |\widehat{g}(\xi_2)|^2}{|\phi'(\xi_1)-\phi'(\xi_2)|}
    \,\d \xi_1\d\xi_2.
\end{equation}
It is valid not only for the fractional Schr\"odinger propagators (i.e. $\phi(\xi)=|\xi|^a$), but more generally the extension operator $E$ for a certain smooth function $\phi: \R\to\R^2$ given by
\[
Ef(x_1,x_2)
=
\int_{\R} e^{i(x_1\xi +x_2\phi(\xi))} f(\xi)\,\d \xi.
\]
Note that the adjoint of $E$ may be viewed as the Fourier restriction operator associated with the curve $\{(\xi,\phi(\xi)): \xi\in\R\}$. 

The bilinear identity \eqref{e:bilinear} is the simplest example of a broader family of multilinear estimates that have emerged as powerful tools in the study of long-standing central problems in harmonic analysis, such as the Fourier restriction conjecture and Kakeya conjecture (\cite{TVV_JAMS,BCT_Acta}). Furthermore, the identity \eqref{e:bilinear} itself has attracted independent attention. It was generalized to higher dimensions by Bennett--Iliopoulou \cite{BennettIliopoulou}, and its interesting connection with X-ray transform is further explored in \cite{BNS_selecta}.

\subsection*{Notation}
To avoid ambiguity, throughout the paper, we use capital letters for functions defined on $\R^2$, and lowercase letters for functions defined on $\R$ or $\T$.
The Fourier transform of a function $f$ is denoted by $\widehat{f}$ or $\mathcal{F}[f]$, and $\mathcal{F}^{-1}$ denotes the inverse Fourier transform.
Regarding constants, for positive quantities $A$ and $B$, the notation $A\lesssim B$ and $A\gtrsim B$ means that $A\le CB$ and $A\ge CB$, respectively, for some positive constant $C$. We write $A\sim B$ if both $A\lesssim B$ and $A\gtrsim B$ hold. When it is important to indicate that the implicit constant depends on a parameter $a$, we use the notation $\lesssim_a$, $\gtrsim_a$, and $\sim_a$.

\subsection*{Overview}

Section~\ref{s:main theorem} is devoted to the proof of Theorem~\ref{t:main}. In Section~\ref{s:periodic}, we apply this result to obtain Strichartz estimates on the torus (Corollary~\ref{c:periodic_Strichartz}). Section~\ref{s:modulation} concerns local smoothing estimates in modulation spaces. It also includes a new orthogonal Strichartz-type estimate (Proposition~\ref{t:orthogonal new}) as well as further application to local well-posedness for the cubic NLS. We also note supplementary material in Appendix, including van der Corput's Lemma and an alternative proof of Theorem~\ref{t:main}.
\section{Proof of Theorem~\ref{t:main}}\label{s:main theorem}
\subsection{Lemma}

For the proof of Theorem~\ref{t:main}, the following van der Corput lemma in sublevel-set form plays a crucial role.

\begin{lem}
\label{lem:overlap}
Let $a\in(0,\infty)\setminus\{1\}$ and $r>0$. For $c\in \R$ and $\lambda>0$, define
\[
U_{a,r}(c,\lambda):=\{t\in[0,r]: t^a+(r-t)^a \in [c-\lambda,c+\lambda]\}.
\]
Then $U_{a,r}(c,\lambda)$ is a union of at most two intervals, and the following estimate holds:
\[
|U_{a,r}(c,\lambda)| \lesssim_a r \min\bigl(1,\lambda^{\frac12} r^{-\frac a2}\bigr).
\]
\end{lem}

\begin{proof}
Define $u:[0,r]\to\R$ by
\[
u(t):=t^a+(r-t)^a-c.
\]
Then
\[
U_{a,r}(c,\lambda)=\{t\in[0,r]: |u(t)|\le \lambda\},
\]
and the second derivative of $u$ is given by
\[
u''(t)=a(a-1)\bigl(t^{a-2}+(r-t)^{a-2}\bigr).
\]
Since $u''$ has constant sign on $(0,r)$, the function $u$ is either convex or concave, and therefore $U_{a,r}(c,\lambda)$ is a sublevel set of a convex/concave function. Consequently, it is a union of at most two intervals.

Finally, using the inequality
\[
A^p+B^p \ge \min(1,2^{-p})(A+B)^p,
\]
which holds for all $A,B>0$ and $p\in\R$, we obtain
\[
|u''(t)|=a|a-1|\bigl(t^{a-2}+(r-t)^{a-2}\bigr)
\gtrsim_a r^{a-2},
\]
uniformly for $t\in(0,r)$. The desired estimate then follows from van der Corput’s lemma in sublevel-set form (Lemma~\ref{l:vandercoroput sublevel}) and the fact that $U_{a,r}(c,\lambda)\subset [0,r]$.
\end{proof}
    
\subsection{Proof of Theorem~\ref{t:main}}

\begin{proof}
    Using the decomposition 
    \[F=\bigl(\sum_{\omega\in (-\infty,- 2\delta^{1/b}]} +\sum_{\omega \in (-2\delta^{1/b},2\delta^{1/b})} + \sum_{\omega\in [ 2\delta^{1/b},\infty)}\bigr) F_\omega\] and the triangle inequality in $L^4$, together with the fact that the square function increases when the number of components increases, we see that it suffices to prove the estimate for functions $F$ satisfying 
    \[\operatorname{supp}\widehat{F}\subset \Gamma_a(\delta)\cap ((0,\infty)\times \R).\]

    Since $\operatorname{supp}\widehat{F_\omega}\subset \theta_\omega$, it follows that 
    \[\operatorname{supp}\bigl(\widehat{F_{\omega}}*\widehat{F_{\omega'}}\bigr)\subset \theta_{\omega}+\theta_{\omega'}.\]
    We then define $N(\delta)$ to be the maximal overlap of the Minkowski sums $\theta_{\omega} + \theta_{\omega'}$, i.e. 
    \[N(\delta):= \Bigl\|\sum_{\omega,\omega'\in \Omega_b(\delta)} 1_{\theta_{\omega}+\theta_{\omega'}}\Bigr\|_{L^{\infty}(\R^2)}.\]
    By Plancherel's theorem, and the Cauchy--Schwarz inequality, we obtain
    \begin{align*}
        \|F\|_{L^4(\R^2)}^4
        &=
        \int_{\R^2}\Bigl|\sum_{\omega\in \Omega_b(\delta)}F_{\omega}\Bigr|^{4}dx 
        =  \int_{\R^2}\Bigl|\sum_{\omega,\omega'\in \Omega_b(\delta)}F_{\omega}F_{\omega'}\Bigr|^2dx
        =
        \int_{\R^2}\Bigl|{\sum_{\omega,\omega'\in \Omega_b(\delta)}\widehat{F_{\omega}}}*\widehat{F_{\omega'}}\Bigr|^2d\xi\\
        & \leq 
        \int_{\R^2} N(\delta) \sum_{\omega,\omega'\in \Omega_b(\delta)} \Bigl|\widehat{F_{\omega}}*\widehat{F_{\omega'}} \Bigr|^2 d\xi
        =
        N(\delta) \sum_{\omega,\omega'\in \Omega_b(\delta)} \int_{\R^2}|F_{\omega} F_{\omega'}|^2 dx\\
        &=
        N(\delta) \Bigl\|\Bigl(\sum_{\omega \in \Omega_b(\delta)} |F_\omega|^2 \Bigr)^{\frac{1}{2}} \Bigr\|_{L^4(\R^2)}^4.
    \end{align*}
    Thus, it suffices to show 
    \[
        N(\delta)
        \lesssim
        \max\left\{
        1,\,
        \delta^{-(\frac1b-\frac 1a)}, \, 
        \delta^{-(\frac1b-\frac 12)}
        \right\}.
    \]
    Suppose that $\xi\in \R^2$ belongs to a Minkowski sum $ \theta_{\omega}+\theta_{\omega'}$ for some $\omega,\omega'\in \Omega_b(\delta)$. Then $\xi_1\in [0,2]$ and there exist $\eta,\eta'\in [0,1]$ such that $|\eta-\omega|\leq \delta^{1/b}$, $|\eta'-\omega'|\leq \delta^{1/b}$ and $\tau, \tau'\in [-\delta,\delta]$ such that
    \[\begin{cases}
            \xi_1=\eta+\eta'\\
            \xi_2=\eta^a+\tau+\eta'^a+\tau'
        \end{cases}.\]
    Define
    \[L(\xi):=\bigl\{(\zeta_1,\zeta_2)\ge 0: \zeta_1+\zeta_2=\xi_1,\; \zeta_1^a+\zeta_2^a\in \xi_2 +[-2\delta,2\delta]\bigr\}\]
    and observe that since $(\eta,\eta')\in L(\xi)$, the point $(\omega,\omega')\in \Omega(\delta)^2$ lies in a $2\delta^{1/b}$-neighborhood of $L(\xi)$.
    Because $\Omega(\delta)^2$ is $\delta^{1/b}$-separated,the disjoint balls of radius $\delta^{1/b}/2$ centered at points of $\Omega_b(\delta)^2\cap(L(\xi)+B(0,2\delta^{1/b}))$ are contained in 
    \[S(\xi) = L(\xi)+B(0,3\delta^{1/b}),\] 
    A volume bound therefore yields
    \[N(\delta)\leq \sup_{\xi:\, \xi_1\in[0,2]} \#\bigl((L(\xi)+B(0,2\delta^{1/b})) \cap \Omega_b(\delta)^2\bigr)\lesssim \sup_{\xi:\, \xi_1\in [0,2]} \delta^{-2/b}|S(\xi)|. \]

    The set $L(\xi)$ is a subset of a line in $\R^2$ and by Lemma~\ref{lem:overlap} with $r=\xi_1$, $\lambda =2\delta$, it is a union of at most two line segments with total length (i.e. the 1-dimensional Hausdorff measure)
    \[\mathcal{H}^1(L(\xi))\lesssim \xi_1\min(1,\delta^{1/2}\xi_1^{-a/2}) 
    =\begin{cases}
        \xi_1, &\quad \xi_1\in [0,\delta^{1/a}]\\
        \xi_1^{1-a/2}\delta^{1/2}, &\quad \xi_1\in [\delta^{1/a},2]
    \end{cases}.\]
    Therefore, 
    \[\sup_{\xi:\, \xi_1\in [0,2]} \mathcal{H}^1(L(\xi))\lesssim_a 
    \begin{cases}
    \delta^{1/a},\quad &a\in[2,\infty)\\
    \delta^{1/2},\quad &a\in (0,2)\setminus\{1\}
    \end{cases}. \]

    Since $L(\xi)$ is a union of at most two line segments, the set $S(\xi)$ is contained in a union of at most two rectangles of width $O(\delta^{1/b})$ and length $\mathcal{H}^1(L(\xi))+O(\delta^{1/b})$. Hence,
    \[|S(\xi)|\lesssim\bigl(\mathcal{H}^1(L(\xi))+O(\delta^{1/b})\bigr)O(\delta^{1/b})\lesssim_a \delta^{2/b}\max \Bigl\{1, \delta^{\frac{1}{a}-\frac{1}{b}}, \delta^{\frac{1}{2}-\frac{1}{b}}\Bigr\}.\]
    Multiplying by $\delta^{-2/b}$ yields the desired bound for $N(\delta)$ and completes the proof.
\end{proof}

\subsection{Sharpness of constants in Theorem \ref{t:main}}

\begin{proof}
    \underline{Case $a\in (0,2]$, $a\neq 1$.}
    
    We need to prove that the best constant in the reverse square function estimate \ref{e:main} satisfies $R_{a,b}(\delta)\gtrsim \max(\delta^{1/2-1/b},1)$.
    Throughout the proof, we assume that $\delta< 1/10$.

    Suppose that $b\in (1,2]$ first. 

    Let $c>0$ be a small number to be chosen later. Define
    \[\Omega_b(\delta):=\delta^{1/b}\Z \cap [1/2-c\delta^{1/2},1/2+c\delta^{1/2}].\]
    and observe that cardinality of this set is  $\#\Omega_b(\delta)\sim_c \delta^{1/2-1/b}$.

    Let $\phi \in \mathcal{S}(\R)$ such that $1_{[-1/4,1/4]}\leq \phi\leq 1_{[-1/2,1/2]}$ and let  $\phi_\delta := \delta^{-1}\phi(\delta^{-1}\cdot)$ denote its $L^1$-normalized rescaling. Define
    \[\widehat{F}(\xi) := \sum_{\omega\in \Omega_b(\delta)}\phi_\delta(\xi_1 - \omega)\phi_\delta(\xi_2-\omega^a).\]

    Observe that we can index the set $\Theta=\{\theta\}$ with the set $\Omega_b$ so that $\theta_\omega \ni (\omega,\omega^a)$.

    Since $b > 1$, we have $\delta < \delta^{1/b}$. Consequently, the $\delta/2\times \delta/2$ support of the bump function centered at $(\omega, \omega^a)$ is contained strictly within $\theta_\omega$, and these supports are mutually disjoint for distinct $\omega$. Thus, we have
    \[\widehat{F}_{\omega}(\xi) = \phi_\delta(\xi_1 - \omega)\phi_\delta(\xi_2-\omega^a).\]
    Let $Q_\delta := [-\delta, \delta]^2$. Note that $\widehat{F}_{\omega}$ is essentially a smoothed indicator function of the rectangle $(\omega, \omega^a) + Q_{\delta/2}$.
    Let
    \[T(\delta):=\bigl(\Omega_b(\delta)+\Omega_b(\delta)\bigr)\cap [1-c\delta^{1/2}/10,1+c\delta^{1/2}/10]\]
    The cardinality is $\# T(\delta)\sim_c \delta^{1/2-1/b}$. Due to the arithmetic progression structure of $\Omega_b(\delta)$, for each fixed $\zeta_1 \in T(\delta)$, there are $\sim_c \delta^{1/2-1/b}$ pairs $(\omega, \omega') \in \Omega_b(\delta)^2$ such that $\omega + \omega' = \zeta_1$.
    For such triple $\zeta_1,\omega,\omega'$, applying Taylor's theorem around $\zeta_1/2$ yields:
    \begin{align*}
    \omega^a + (\omega')^a &= \omega^a + (\zeta_1 - \omega)^a
    = 2^{1-a}\zeta_1^a + O_a\left((\omega - \zeta_1/2)^2\right)\\
    &=2^{1-a}\zeta_1^a+O_a(c^2\delta),
    \end{align*}
    with the implicit constant independent of $\zeta_1$, $c$ and $\delta$. By choosing $c>0$ small enough, the following estimate holds
    \[(\omega,\omega^a)+(\omega',\omega'^a) \in (\zeta_1,2^{1-a}\zeta_1^{a})+\{0\}\times [-\delta/10,\delta/10].\]
    Using the properties of convolution of indicator functions, together with the previous inequality, we get the following inequality
    \begin{align*}
        \widehat{F}_{\omega}*\widehat{F}_{\omega'}(\xi) &\gtrsim  |Q_\delta|^{-2}\bigl( 1_{(\omega,\omega^a)+Q_{\delta/4}}*1_{(\omega',\omega'^a)+Q_{\delta/4}}\bigr)(\xi)
        \gtrsim |Q_\delta|^{-1}1_{(\zeta_1,2^{1-a}\zeta_1^a)+Q_{\delta/8}}(\xi).
    \end{align*}
    Summing the previous inequality over all $\sim \delta^{1/2-1/b}$ triples $(\xi_1,\omega,\omega')\in T(\delta)\times \Omega_b(\delta)^2$ for which $\omega+\omega'=\zeta_1$ holds, we conclude the following inequality holds
    \begin{align*}
        \int_{\R^2} \Bigl|\sum_{\omega_1,\omega_2} \widehat{F}_{\omega_1}*\widehat{F}_{\omega_2}(\xi) \Bigr|^2 d\xi 
        &\gtrsim \int_{\R^2} \Bigl| \sum_{\zeta_1\in T(\delta)} \delta^{1/2-1/b}|Q_\delta|^{-1}1_{(\zeta_1,2^{1-a}\zeta_1^a)+Q_{\delta/8}}(\xi)\Bigr|^2 d\xi\\
        &= ( \delta^{1/2-1/b}|Q_{\delta}|^{-1})^2 \int_{\R^2}\sum_{\zeta_1\in T(\delta)} 1_{(\zeta_1,2^{1-a}\zeta_1^a)+Q_{\delta/8}}(\xi) d\xi\\
        &\sim \delta^{3(1/2-1/b)} |Q_{\delta}|^{-1}.
    \end{align*}
    On the other hand,
    \begin{align*}
        \int_{\R^2} \sum_{\omega_1,\omega_2} |\widehat{F}_{\omega_1}*\widehat{F}_{\omega_2}(\xi)|^2 d\xi\lesssim \int_{\R^2} \sum_{\omega_1,\omega_2\in \Omega_b}|Q_\delta|^{-2}1_{(\omega_1+\omega_2, \omega_1^a+\omega_2^a)+Q_\delta}(\xi) d\xi\sim \delta^{2(1/2-1/b)}|Q_\delta|^{-1}.
    \end{align*}
    Comparing the expressions, we conclude that  $R_{a,b}(\delta)\gtrsim\delta^{1/2-1/b}$.

    Consider now the case $b<1$. In this regime, we define:
    \[
    \widehat{F}(\xi) = \sum_{\omega \in \Omega_b} \phi_{\delta^{1/b}}(\xi_1 - \omega)\phi_{\delta}(\xi_2 - \omega^a).
    \]
    Since $b < 1$, we have $\delta^{1/b} < \delta$. The frequency support around $(\omega, \omega^a)$ is now a rectangle of dimensions $\delta^{1/b} \times \delta$. This rectangle is still a subset of the sector $\theta_\omega$, and mutual disjointness is preserved. The calculation proceeds identically to Case 1, replacing the square $Q_\delta$ with the rectangle $\Tilde{Q}_\delta = [-\delta^{1/b}, \delta^{1/b}] \times [-\delta, \delta]$. The volume factors cancel out, giving the same lower bound.

    Finally, when $b\ge 2$, we simply select a single wave packet:
    \[
    \widehat{F}(\xi) = \phi_\delta(\xi_1)\phi_\delta(\xi_2).
    \]
    The sum over $\theta$ trivializes to a single term. The left-hand side and right-hand side of the estimate become identical, yielding $R_{a,b}(\delta)\gtrsim 1$.

    \underline{Case $a > 2$.}
    We need to prove that the best constant in the reverse square function estimate \ref{e:main} satisfies $R_{a,b}(\delta)\gtrsim \max(\delta^{1/a-1/b},1)$.

    If $a<b$, the estimate follows by choosing $f$ to be equal to a single wave packet
    \[\widehat{F}(\xi)=\phi_{\delta^{\max(1,1/b)}}(\xi_1)\phi_\delta(\xi_2).\]

    When $a\ge b$, let $c>0$ be a small number to be chosen later and define
    \[\Omega_b(\delta):= \delta^{1/b}\Z \cap [0,c\delta^{1/a}]\]
    and observe that cardinality of this set is $\#\Omega_b(\delta)\sim \delta^{1/a-1/b}$.
    Similarly to the first part, when $b\ge 1$, we define 
    \[\widehat{F}(\xi)=\sum_{\omega\in \Omega_b(\delta)}\phi_\delta(\xi_1-\omega)\phi_\delta(\xi_2-\omega^a).\]
    Then, if we define
    \[T(\delta):= \bigl(\Omega_b(\delta)+\Omega_b(\delta)\bigr)\cap [c\delta^{1/a}/2,3c\delta^{1/a}2],\]
    the cardinality is $\#T(\delta)\sim \delta^{1/a-1/b}$ and for each $\zeta_1\in T(\delta)$ there exist $\sim \delta^{1/a-1/b}$ pairs $(\omega,\omega')\in \Omega_b(\delta)^2$ such that $\omega+\omega'=\zeta_1$. For each such triplet, by the trivial estimate, we have
    \[|\omega^a+\omega'^a|\leq |\omega^a| + |\omega|^a\leq 2c^a\delta.\]
    When $c>0$ is small enough, we can see that one can repeat the same proof as in the first part, because there is a significant overlap on squares $(\zeta_1,0)+Q_{\delta/8}$, for $\zeta_1\in T(\delta)$.

    When $b<1$, we modify the proof in the same way as in the first case.
\end{proof}

\section{Proof of Periodic Strichartz estimate}\label{s:periodic}
\subsection{Sufficiency}
\begin{proof}[Proof of Corollary \ref{c:periodic_Strichartz}]
Suppose first that $f$ is a periodic function such that $\operatorname{supp}\widehat{f}\subset [-N,N]$.

Let $\eta\in \mathcal{S}(\R)$ be a Schwartz function such that $\eta|_{[-1/2,1/2]}\gtrsim 1$ $\operatorname{supp} \widehat{\eta}\subset [-1/2,1/2]$. Define $\eta_t(x)=t^{-1}\eta(t^{-1}x)$.
Using the change of variables $x\mapsto x/N$ and $t\mapsto t/N^a$, we obtain
\begin{align*}
    \bigl\|e^{2\pi it(-\partial_x^2)^{\frac a2}}f\bigr\|_{L^4_{t,x}([-1,1]\times \T)}^4 
    &=N^{-(a+1)}\int_{-N/2}^{N/2}\int_{-(N/2)^a}^{(N/2)^a}\Bigl|\sum_{k\in \Z, |k|\leq N} \widehat{f}(k)e^{2\pi i (|\frac{k}{N}|^at + \frac{k}{N}x)}\Bigr|^4dtdx=(*).
\end{align*}
When $a<1$, using the properties of the function $\eta$, we obtain the estimate
 \[(*) \lesssim N^{-(a+1)} \int_{\R^2}\Bigl|\sum_{k\in \Z, |k|\leq N} \widehat{f}(k)e^{2\pi i (|\frac{k}{N}|^at + \frac{k}{N}x)}(\eta_{N}\otimes \eta_{N^a})(x,t)\Bigr|^4dxdt.\]
When $a>1$, we first observe that for an $N$-periodic function $g:\R\to\C$ the estimate
    \[N^{-1}\int_{-N/2}^{N/2} |g| \lesssim N^{-a}\int_{-(N/2)^a}^{,(N/2)^a} |g|\]
    holds, so using the properties of the function $\eta$ we obtain
    \begin{align*}
        (*)
        &\lesssim N^{-2a} \int_{\R^2}\Bigl|\sum_{k\in \Z, |k|\leq N} \widehat{f}(k)e^{2\pi i (|\frac{k}{N}|^at + \frac{k}{N}x)}(\eta_{N^a}\otimes \eta_{N^a})(x,t)\Bigr|^4dxdt.
    \end{align*}
    Define
    \[F_{a,k}(x,t):=N^{-\max(1,a)-a} \widehat{f}(k)e^{2\pi i (|\frac{k}{N}|^at + \frac{k}{N}x)}(\eta_{N^{\max(1,a)}}\otimes \eta_{N^a})(x,t).\]
    Using the frequency localization property of the function $\eta$,
    \[\operatorname{supp} \widehat{F_{a,k}} \subset ( kN^{-1},|kN^{-1} |^a )+\frac{1}{2}\Bigl([-N^{-\max(1,a)},N^{-\max(1,a)}]\times [-N^{-1},N^{-1}]\Bigr)\subset \theta_k,\]
    where $\theta_k$ is the $k$-th $N^{-1}\times N^{-a}$ neighborhood in Theorem \ref{t:main} \footnote{This is why we used periodicity to get better frequency localization when $a>2$ because otherwise the $N^{-1}\times N^{-a}$ rectangle would not be a subset of $\theta_k$.}
    
    Applying Theorem \ref{t:main}, with $b=a$ and $\delta = N^{-a}$, we obtain
\[
    \Bigl\|\sum_{k}F_{a,k}\Bigr\|_{L^4_{t,x}}
    \lesssim  N^{\max\{0,\frac{1}{4}(1-\frac{a}{2})\}}\Bigl\| \Bigl(\sum_{k} |F_{a,k}|^2 \Bigr)^{\frac{1}{2}} \Bigr\|_{L^4_{t,x}}.
\]
Finally, observing that
\[\Bigl\| \Bigl(\sum_{k} |F_{a,k}|^2 \Bigr)^{\frac{1}{2}} \Bigr\|_{L^4_{x,t}}\\ 
     = N^{-\frac{\max(1,a)+a}{4}} \Bigl\| \Bigl( \sum_{k} |\widehat{f}(k)|^2 \Bigr)^{\frac{1}{2}} \eta_{N^{\max(1,a)}}\otimes \eta_{N^a}\Bigr\|_{L^4_{t,x}} 
    \lesssim \|f\|_{L^2(\T)},\]
the claimed estimate holds for frequency localized $f$.

For a general function $f$, we use the Littlewood--Paley theory. Denote with $\Delta_j$ the frequency projection to $\sim 2^j$ and with $\Delta_{\leq 0}$ projection to $\lesssim 1$. Then \[\Delta_j(e^{it(-\partial_x^2)^{a/2}}f)= e^{it(-\partial_x^2)^{a/2}} (\Delta_jf).\] Therefore, using Littlewood--Paley and Minkowski inequality, together with the first part of the proof, we obtain

\begin{align*}
    \|e^{2\pi it(-\partial_x^2)^{\frac a2}}f\|_{L^4_{t,x}([-1,1]\times \T)} 
    &\sim \Bigl\| \Bigl( |e^{2\pi it(-\partial_x^2)^{\frac a2}}(\Delta_{\leq 0}f)|^2+\sum_{j\ge 0}|e^{2\pi it(-\partial_x^2)^{\frac a2}}(\Delta_jf)|^{2}\Bigr)^{\frac{1}{2}} \Bigr\|_{L^4_{t,x}}\\
    &\lesssim \Bigl( \bigl\|e^{2\pi it(-\partial_x^2)^{\frac a2}}(\Delta_{\leq 0}f)\bigr\|_{L^4_{t,x}}^2 + \sum_{j\ge 0} \bigl\|e^{2\pi it(-\partial_x^2)^{\frac a2}}(\Delta_jf)\bigr\|_{L^4_{t,x}}^2 \Bigr)^{\frac{1}{2}}\\
    &\lesssim\Bigl( \|(\Delta_{\leq 0}f)\|_{L^2_x}^2 + \sum_{j\ge 0} 2^{j\max\{0,\frac{1}{4}(1-\frac{a}{2})\}}\|\Delta_jf\|_{L^2_x}^2 \Bigr)^{\frac{1}{2}}\\
    &\sim \begin{cases}
        \|f\|_{L^2_x(\T)},\quad &a\ge 2\\
        \|f\|_{H^s_x(\T)},\quad &a\in (0,1)\cup (1,2), \; s=\frac{1}{4}(1-\frac{a}{2}).
    \end{cases}
\end{align*}
\end{proof}

\subsection{Necessity}
Observe that the only inequality used in the proof of Corollary \ref{c:periodic_Strichartz} is the reverse square function estimate \ref{e:main} in the case \(a=b\). Therefore, if we choose the Schwartz function \(\eta \in \mathcal{S}(\R)\) in that proof so that
$1_{[-1/4,\,1/4]} \le \widehat{\eta} \le 1_{[-1/2,\,1/2]}$,
then all steps in the preceding argument, except for the application of the reverse square function estimate, become equalities up to multiplicative constants. Moreover, with this choice of \(\eta\), the function $f$ coincides with the one used in the proof of sharpness of the reverse square function estimate. Consequently, the inequality arising from that step is also an equality up to a multiplicative constant, so the constants in Corollary \ref{c:periodic_Strichartz} are sharp.

\section{Local smoothing estimates in modulation spaces}\label{s:modulation}
In this subsection, we present the proof of Corollary \ref{c:local_smoothing_modulation} and Corollary \ref{c:local_smoothing_modulation_lu}.
In both proofs, we use the following reverse square function estimate for extension operators. 
\begin{prop}
    Let $R \ge 1$ and $a\geq 2$. For $f\in \mathcal{S}(\R)$, the estimate
    \[
    \bigl\|e^{{it(-\partial^2_x)}^{\frac a2}}f(x)\bigr\|_{L^{4}_{t,x}(B_{R^a})} \lesssim \Bigl\|\Bigl(\sum_{k\in\Z}\bigl|e^{{it(-\partial^2_{x})}^{\frac a2}}\square_k^{\sharp}f(x)\bigr|^2\Bigr)^{1/2}\Bigr\|_{L^{4}_{t,x}(w_{B_{R^a}})}
    \]
    holds where $\widehat{\square_k^{\sharp}h}(\xi) := 1_{
    [R^{-1}(k-1/2), R^{-1}(k +1/2)]}(\xi)\widehat{h}(\xi)$, 
    $B_{R^a}$ denotes arbitrary two dimensional balls with radius $R^a$, and $w_{B_{R^a}}$ denotes a weight function that rapidly decays off outside $B_{R^a}$. \label{p:RSF_extension}
\end{prop}

\begin{proof}
    Let $\eta_{B_{R^a}}\in \mathcal{S}(\R^{2})$ satisfy
    \begin{align*}
        \begin{cases}
            \supp\mathcal{F}[\eta^{\frac{1}{4}}_{B_{R^a}}] \subset B_{2}(0,R^{-a})\\
            \eta_{B_{R^a}} \gtrsim 1 \thinspace \mathrm{on}\thinspace B_{2}(0,R^{a})
        \end{cases}
    \end{align*}
    where $B_{2}(0,K)$ denotes a two dimensional ball centered at the origin and with radius $K>0$. 
    Note that $\supp \F_{t,x}[e^{{it(-\partial^2_x)}^{\frac a2}} f\cdot \eta_{B_{R^a}}^{\frac{1}{4}}]$ is contained in a $R^{-a}$-neighborhood of the curve $(\xi, |\xi|^a)$. Thus, we can apply Theorem \ref{t:main} as follows. 
    \begin{align*}
    \bigl\|e^{{it(-\partial^2_x)}^{\frac a2}}f(x)\bigr\|_{L^{4}_{t,x}(B_{R^a})} &\lesssim \bigl\|e^{{it(-\partial^2_x)}^{\frac a2}}f(x)\eta^{\frac{1}{4}}_{B_{R^a}}\bigr\|_{L^{4}_{t,x}(\R^2)}\\
    &\lesssim \Bigl\|\Bigl(\sum_{k\in\Z}\bigl|\square_k^{\sharp}(e^{{it(-\partial^2_{x})}^{\frac a2}}f(x)\eta^{\frac{1}{4}}_{B_{R^a}})\bigl|^2\Bigr)^{1/2}\Bigr\|_{L^{4}_{t,x}(\R^2)}. 
    \end{align*}
   For each $k$, let $S_k$ be a set of indices that satisfy $\square_k^{\sharp}\Bigl(\sum_{i\in S_{k} }e^{{it(-\partial^2_{x})}^{\frac a2}}\square_i^{\sharp}f(x)\eta^{\frac{1}{4}}_{B_{R^a}}\Bigr)\neq 0$. Note that the number of components of $S_k$ is independent of $R$ because the multiplication by $\eta_{B_{R^a}}$ enlarges the Fourier support only by $R^{-a}$. Hence, we have
    \[
    \square_k^{\sharp}(e^{{it(-\partial^2_{x})}^{\frac a2}}f(x)\eta^{\frac{1}{4}}_{B_{R^a}}) = \square_k^{\sharp}\Bigl(\sum_{i \in S_{k}}e^{{it(-\partial^2_{x})}^{\frac a2}}\square_i^{\sharp}f(x)\eta^{\frac{1}{4}}_{B_{R^a}}\Bigr). 
    \]
     Then by using the boundedness of the Hilbert transform and the inequality $\eta_{B_{R^a}} \lesssim w_{B_{R^a}}$, we obtain the desired result. 
\end{proof}

We also collect some useful lemmas here.
\begin{prop}[{\cite[Theorem 1]{chaichenets2020local}}]
    Let $d\in\mathbb{N}$, $p,q\in [1,\infty]$, $s\in\R$. Then \begin{equation*}
        \|f\|_{M^{s}_{p,q}(\R^d)} \sim \bigl\|2^{sj}\|\Delta_{j}f\|_{M_{p,q}(\R^d)}\bigr\|_{\ell^{q}_{j}}
    \end{equation*}
    where $\{\Delta_{j}\}$ denotes the Littlewood-Paley decomposition. \label{thm:littlewoodpaley}
\end{prop}

\begin{prop}
    Let $\theta\in (0,1)$, $p_{1}, p_{2},q_{1}, q_{1}\in [1,\infty)$, and $s_{1}, s_2 \in \R$. Set $1/p = (1-\theta)p_1 + \theta p_2$, $1/q = (1-\theta)q_1 + \theta q_2$, and $s = (1-\theta)s_1 + \theta q_2$. Then
    \[
    (M^{s_1}_{p_1, q_1}, M^{s_2}_{p_2, q_2})_{[\theta]} = M^s_{p,q. }
    \]
\end{prop}

\subsection{Proof of Corollary \ref{c:local_smoothing_modulation}}
For Proposition \ref{t:orthogonal new}, we need to show the following lemma. Although this lemma is a consequence of Theorem 1.1 in Bennett--Iliopoulou \cite{BennettIliopoulou}, we present its proof here for reader's convenience. 

\begin{lem}[Bilinear Strichartz estimate]
    Let $a \ge 2$ and let $k,j\in\Z$ with $|k - j| > 4$, $k\cdot j > 0$, and $|k|,|j| \ge 2$. Suppose $\supp\wdh{f}\subset [k -1, k+1]$ and $\supp\wdh{g}\subset[j - 1, j+1]$. Then the inequality
    \begin{equation*}
        \bigl\|[e^{it(-\partial_x^2)^{\frac a2}}f(x)][e^{it(-\partial_x^2)^{\frac a2}}g(x)]\bigr\|_{L^{2}_{t,x}(\R^2)} \lesssim 
        \ang{k-j}^{\frac{1 - a}{2}}\|f\|_{\ls{2}{x}{\R}}\norm{g}_{\ls{2}{x}{\R}}
    \end{equation*}
    holds. \label{lem:bilinear_strichartz}
\end{lem}

\begin{proof}[Proof of Lemma \ref{lem:bilinear_strichartz}]
    If $k,j < 0$, we change the variables and reduce this case to the case $k,j > 0$. Hence, we may assume that $k,j$ are positive. We also may assume that $k > j$. Note that 
    \begin{equation*}
        \bigl\|[e^{it(-\partial_x^2)^{\frac a2}}f(x)][e^{it(-\partial_x^2)^{\frac a2}}g(x)]\bigr\|_{L^{2}_{t,x}(\R^2)} = \bigl\|[e^{it(-\partial_x^2)^{\frac a2}}f(x)]\overline{[e^{it(-\partial_x^2)^{\frac a2}}g(x)]}\bigr\|_{L^{2}_{t,x}(\R^2)}. 
    \end{equation*}
    Since the Fourier supports of $f$ and $g$ are contained in $[0,\infty)$ and $k > j$, we have
    \begin{align*}
        [e^{it(-\partial_x^2)^{\frac a2}}f(x)]\overline{[e^{it(-\partial_x^2)^{\frac a2}}g(x)]} &= \frac{1}{(2\pi)^2}\int_{\{\xi > \eta > 0\}}e^{i(t\xi^a + x\xi)}\wdh{f}(\xi)e^{-i(t\eta^a + x\eta)}\overline{\wdh{g}(\eta)}d\xi d\eta\\
        &=\frac{1}{(2\pi)^2}\int_{\{\xi > \eta > 0\}}e^{i[(\xi - \eta)x + (\xi^a - \eta^a)t]}\wdh{f}(\xi)\overline{\wdh{g}(\eta)}d\xi d\eta
    \end{align*}
    We apply the change of variables defined by
    \begin{equation*}
    T:(\xi,\eta)\mapsto (\zeta_1,\zeta_2),
    \end{equation*}
    where
    \begin{equation*}
    \zeta_1=\xi-\eta,\qquad \zeta_2=\xi^{a}-\eta^{a}.
    \end{equation*}
    We denote the Jacobian of this map by
    \[
    J_T(\xi,\eta):=\det DT(\xi,\eta)
    = \det
    \begin{pmatrix}
    1 & -1 \\
    a\xi^{\,a-1} & -a\eta^{\,a-1}
    \end{pmatrix}
    = a(\xi^{\,a-1}-\eta^{\,a-1})
    \]
    the Jacobian determinant of this transformation. With this notation, we obtain
    \begin{equation*}
     \frac{1}{(2\pi)^2}\int_{\{\xi > \eta > 0\}} 
       e^{i[(\xi - \eta)x + (\xi^a - \eta^a)t]}
       \widehat{f}(\xi)\,\overline{\widehat{g}(\eta)}\,d\xi d\eta
     = \frac{1}{(2\pi)^2}\int_{\Omega}
       e^{i(\zeta_{1}x +\zeta_{2}t)}
       \frac{[\widehat{f}\otimes \overline{\widehat{g}}]\circ T^{-1}(\zeta_{1}, \zeta_{2})}
            {|J_T\circ T^{-1}(\zeta_{1}, \zeta_{2})|}
       d\zeta_{1} d\zeta_{2} .
    \end{equation*}
    where $\Omega := T(\{\xi > \eta > 0\})$. Set 
    \begin{equation*}
        F(\zeta_{1}, \zeta_{2}) := 1_{\Omega}(\zeta_{1}, \zeta_{2})\frac{[\widehat{f}\otimes \overline{\widehat{g}}]\circ T^{-1}(\zeta_{1}, \zeta_{2})}{|J_T\circ T^{-1}(\zeta_{1}, \zeta_{2})|}. 
    \end{equation*}
    Then note that the right-hand side of the last integral is the Fourier transform of $F$. Thus, by the Plancherel theorem, one obtains 
    \begin{equation*}
         \|[e^{it(-\partial_x^2)^{\frac a2}}f(x)][e^{it(-\partial_x^2)^{\frac a2}}g(x)]\|_{L^{2}_{t,x}(\R^2)} \sim \|F\|_{\ls{2}{\zeta_{1},\zeta_{2}}{\R^2}}. 
    \end{equation*}
    Hence, by changing the variables again, we have
    \begin{align*}
        \|F\|^{2}_{\ls{2}{\zeta_{1},\zeta_{2}}{\R^2}} &= \int_{\R^2}1_{\Omega}(\zeta_{1}, \zeta_{2})\frac{|[\widehat{f}\otimes \overline{\widehat{g}}]\circ T^{-1}(\zeta_{1}, \zeta_{2})|^{2}}{|J_T\circ T^{-1}(\zeta_{1}, \zeta_{2})|^2}d\zeta_{1} d\zeta_{2}\\
        &= \int_{\{\xi > \eta > 0\}}\frac{|\wdh{f}(\xi)|^2|\wdh{g}(\eta)|^2}{|J_{T}(\xi , \eta)|^2}|J_{T}(\xi,\eta)|d\xi d\eta\\
        &= \int_{\{\xi > \eta > 0\}}|\wdh{f}(\xi)|^2|\wdh{g}(\eta)|^2|J_{T}(\xi,\eta)|^{-1}d\xi d\eta
    \end{align*}
    It is easy to verify that $|\xi^{a-1} - \eta^{a-1}| \gtrsim |\xi - \eta|^{a-1} \gtrsim_{a} |k - j|^{a-1}$. Therefore we obtain 
    \begin{equation*}
         \bigl\|e^{it(-\partial_x^2)^{\frac a2}}f(x)][e^{it(-\partial_x^2)^{\frac a2}}g(x)]\bigr\|_{L^{2}_{t,x}(\R^2)} \lesssim_{a} \ang{k-j}^{1-a}\|f\|_{\ls{2}{}{\R}}\|g\|_{\ls{2}{}{\R}}. 
    \end{equation*}
    \end{proof}

\begin{proof}[Proof of Proposition \ref{t:orthogonal new}]
    By expanding the left-hand side, we have
    \begin{align*}
        \mathrm{LHS}^4 &\lesssim \Bigl\|\Bigl(\sum_{k\ge 2}|e^{it(-\partial_x^2)^{\frac a2}}\sq{k}f(x)|^2\Bigr)^{\frac{1}{2}}\Bigr\|^4_{L^{4}_{t,x}(w_{I} \times\R)} + \Bigl\|\Bigl(\sum_{k\le -2}|e^{it(-\partial_x^2)^{\frac a2}}\sq{k}f(x)|^2\Bigr)^{\frac{1}{2}}\Bigr\|^4_{L^{4}_{t,x}(w_{I} \times\R)}\\
        &\quad +  \Bigl\|\Bigl(\sum_{|k|\le 2}|e^{it(-\partial_x^2)^{\frac a2}}\sq{k}f(x)|^2\Bigr)^{\frac{1}{2}}\Bigr\|^4_{L^{4}_{t,x}(w_{I} \times\R)}\\
        &\lesssim
        \sum_{k\in\Z}\int_{\R\times\R}|e^{it(-\partial_x^2)^{\frac a2}}\sq{k}f(x)|^4 w_{I}(t)dtdx\\
        &
        \quad +\sum_{j = 1,2,3}\sum_{\substack{k,k'\in 4\N + j \\ k\neq k'}}\int_{\R\times\R}|e^{it(-\partial_x^2)^{\frac a2}}\sq{k}f(x)|^2 |e^{it(-\partial_x^2)^{\frac a2}}\sq{k'}f(x)|^2 w_{I}(t)dtdx\\
        &\quad + \sum_{j = 1,2,3}\sum_{\substack{-k,-k'\in 4\N + j \\ k\neq k'}}\int_{\R\times\R}|e^{it(-\partial_x^2)^{\frac a2}}\sq{k}f(x)|^2 |e^{it(-\partial_x^2)^{\frac a2}}\sq{k'}f(x)|^2 w_{I}(t)dtdx
    \end{align*}
    For the first term, we apply the Bernstein (Nikolskii) inequality, and 
    for the second term and the third term, we apply \ref{lem:bilinear_strichartz}. Then, we obtain
    \begin{align*}
        \mathrm{LHS}^4 &\lesssim \bigl\|\|\sq{k}f\|_{L^{2}_{x}(\R)}\bigr\|^{4}_{\ell^4_{k}}\\
        &
        \quad +\sum_{k,k'\in \Z}\ang{k-k'}^{1-a}\int_{w_{I}\times\R}|e^{it(-\partial_x^2)^{\frac a2}}\sq{k}f(x)|^2 |e^{it(-\partial_x^2)^{\frac a2}}\sq{k'}f(x)|^2 dtdx. 
    \end{align*}
    Applying the Young inequality, we obtain
    \[
    \mathrm{LHS}^4 \lesssim \bigl\|\|\sq{k}f\|_{L^{2}_{x}(\R)}\bigr\|^{4}_{\ell^4_{k}}. 
    \]
    This is the desired result. 
\end{proof}

Now, we are ready to prove Corollary \ref{c:local_smoothing_modulation}. 

\begin{proof}[Proof of Corollary \ref{c:local_smoothing_modulation}]
    Since we have the Littlewood-Paley characterization for modulation spaces, it suffices to consider the case that $f\in \mathcal{S}(\R)$ and $\supp\wdh{f}\subset B_{1}(0,\lam)$ for $\lam \ge 1$. Set $g(x) := f(x/\lam)$ so that $\supp\wdh{g}\subset B_{1}(0,1)$. Then 
    \begin{equation*}
        \bigl\|e^{it(-\partial_x^2)^{\frac a2}}f(x)\bigr\|_{L^{4}_{t,x}(I \times \R)} = \lam^{-\frac{a +1}{4}}\bigl\|e^{it(-\partial_x^2)^{\frac a2}}g(x)\bigr\|_{L^{4}_{t,x}(\lam^aI \times \R)}
    \end{equation*}
    where $\lam^a I := [-\lam^a , \lam^a]$. We decompose $\R$ into finitely overlapping balls with radius $\lambda^a$. Let $\{B_{R^a}\}$ denote these balls. Then we have
    \begin{equation*}
        \bigl\|e^{it(-\partial_x^2)^{\frac a2}}g(x)\bigr\|_{L^{4}_{t,x}(\lam^aI \times \R)} \leq \Bigl(\sum_{B_{\lam^a}}\bigl\|e^{it(-\partial_x^2)^{\frac a2}}g(x)\bigr\|^4_{L^{4}_{t,x}(\lam^aI \times B_{\lam^a})}\Bigr)^{1/4}. 
    \end{equation*}
    Since $\wdh{g}\subset B_{1}(0,1)$, we apply Proposition \ref{p:RSF_extension} to the right-hand side of this inequality. It follows that 
    \[
    \bigl\|e^{it(-\partial_x^2)^{\frac a2}}g(x)\bigr\|_{L^{4}_{t,x}(\lam^aI \times B_{R\lam^a})} \leq \Bigl\|\Bigl(\sum_{k\in\Z}|e^{it(-\partial_x^2)^{\frac a2}}\square_k^{\sharp}g(x)|^2\Bigr)^{\frac{1}{2}}\Bigr\|_{L^{4}_{t,x}(w_{\lam^aI} \times w_{B_{\lam^a}})}
    \]
    where $\widehat{\square_k^{\sharp}h}(\xi) := 1_{
    [\lambda^{-1}(k-1), \lambda^{-1}(k +1)]}(\xi)\widehat{h}(\xi)$ and $w_{B_{\lam^a}}$ denotes the weight rapidly decays off outside $B_{\lam^a}$. We use rescaling as $x\to\lam x$,  $t\to \lam^a t$, and $\sum_{B_{\lam^a}}w_{B_{\lam^a}} \lesssim 1$ Then, one has
    \[
    \bigl\|e^{it(-\partial_x^2)^{\frac a2}}f(x)\bigr\|_{L^{4}_{t,x}(I \times \R)} \lesssim \Bigl\|\Bigl(\sum_{k\in\Z}|e^{it(-\partial_x^2)^{\frac a2}}\sq{k}f(x)|^2\Bigr)^{\frac{1}{2}}\Bigr\|_{L^{4}_{t,x}(w_{I} \times\R)}
    \]
    where $\sq{k}$ are the same Fourier multiplier as in the definition of modulation spaces. 
    Last, For \eqref{eq:localsmoothing_BLN}, we apply Theorem \ref{t:BLN} to the right-hand side and get the desired result. For \eqref{eq:localsmoothing_Bilinear}, we apply Proposition \ref{t:orthogonal new} instead. 
\end{proof}

\subsection{Proof of Corollary \ref{c:local_smoothing_modulation_lu}}
In a previous work \cite{Lu2023}, a time-fixed estimate for $e^{it(-\partial_x^2)^{\frac a2}}$ was employed, which led to a regularity loss. By replacing this step with a local smoothing estimate, we are able to avoid this derivative loss. To prove our local smoothing estimate, we show the following time-weighted local smoothing estimate. 
\begin{prop}
    Let $\varepsilon > 0$ and $a\ge 2$. Suppose that $\eta_{I}\in\mathcal{S}(\R^d)$ is a nonnegative function and satisfies $\eta_{I} \gtrsim 1$ on $I$. Then the following estimate holds: 
    \begin{equation*}
        \|e^{it(-\partial_x^2)^{\frac a2}}f(x)\|_{L^{4}_{t,x}(\eta_{I} \times \R)}\lesssim_{\varepsilon} \|f\|_{W^{\varepsilon,4}(\R)}. 
    \end{equation*}
    \label{prop:weighted_local_smoothing}
\end{prop}

To show the time-weighted local smoothing, we need to use the following two results. 
\begin{thm}[{Theorem 1.2. in \cite{rogers_seeger2010}}]
    Let $p\in \Bigl(2 + \frac{4}{d + 1}, \infty\Bigr)$ and $a > 1$. Then, for any compact time interval $I$, 
    \begin{equation}
        \bigl\|e^{it(-\partial_x^2)^{\frac a2}}f(x)\bigr\|_{L^{p}_{t,x}(I\times \R^d)} \leq C_{I,p,a}\|f\|_{W^{\beta, p}(\R^d)}, \quad \frac{\beta}{a} = d\Bigl(\frac{1}{2} - \frac{1}{p}\Bigr) - \frac{1}{p}. \label{eq:rogers_seeger}
    \end{equation}
    \label{thm:rogers_seeger}
\end{thm}

\begin{lem}
    Let $p\geq 2$ and $s > 0$. Suppose that $\eta_{I}\in\mathcal{S}(\R^d)$ is a nonnegative function and satisfies $\eta_{I} \gtrsim 1$ on $I$. If the inequality 
    \begin{equation}
        \bigl\|e^{it(-\partial_x^2)^{\frac a2}}f(x)\bigr\|_{L^{p}_{t,x}(I \times\R)} \lesssim \|f\|_{W^{s,p}(\R)} \label{eq:nowieght}
    \end{equation}
    holds, then one also has
    \begin{equation}
        \bigl\|e^{it(-\partial_x^2)^{\frac a2}}f(x)\bigr\|_{L^{p}_{t,x}(\eta_{I} \times\R)} \lesssim \|f\|_{W^{s,p}(\R)}. \label{eq:weighted}
    \end{equation} \label{lem:nonwetght_to_weight}
\end{lem}

\begin{proof}
    Note that by a change of variables we see that \eqref{eq:nowieght} becomes
     \begin{equation*}
        \bigl\|e^{it(-\partial_x^2)^{\frac a2}}f(x)\bigr\|_{L^{p}_{t,x}([-k,k] \times\R)} \lesssim \ang{k}^{1/a}\|f\|_{W^{s,p}(\R)}. 
    \end{equation*}
    Since $\eta_{I}\in\mathcal{S}(\R)$, for each $N>0$ we have
    \begin{equation*}
        \eta_I(t) \lesssim_N (1+|k|)^{-N}, \qquad t\in[k-\tfrac12,k+\tfrac12].
    \end{equation*}
    Hence, for $f\in\mathcal{S}(\R)$,
    \begin{align*}
        \bigl\|e^{it(-\partial_x^2)^{\frac a2}}f(x)\bigr\|_{L^{p}_{t,x}(\eta_{I} \times\R)} &= \Bigl(\int_{\R}\int_{\R}\bigl|e^{it(-\partial_x^2)^{\frac a2}}f(x)\bigr|^p \eta_I (t)dtdx\Bigr)^{\frac{1}{p}}\\
        &= \Bigl(\sum_{k\in\Z}\int_{\bigl[k - \frac{1}{2}, k + \frac{1}{2}\bigr]}\int_{\R}\bigl|e^{it(-\partial_x^2)^{\frac a2}}f(x)\bigr|^p \eta_I (t)dtdx\Bigr)^{\frac{1}{p}}\\
        &\lesssim_N \Bigl(\sum_{k\in\Z}(1 + |k|)^{-N}\int_{\bigl[k - \frac{1}{2}, k + \frac{1}{2}\bigr]}\int_{\R}\bigl|e^{it(-\partial_x^2)^{\frac a2}}f(x)\bigr|^p dtdx\Bigr)^{\frac{1}{p}}\\
        &\lesssim\sum_{k\in\Z}(1 + |k|)^{-N}k^{\frac{1}{a}}\|f\|_{W^{s,p}(\R)}.  
    \end{align*}
     Choosing $N$ sufficiently large, the series converges, and we obtain the desired bound.
\end{proof}

\begin{proof}[Proof of Proposition \ref{prop:weighted_local_smoothing}]
    By interpolating the one-dimensional case of \eqref{eq:rogers_seeger} with the trivial inequality 
    \begin{equation*}
        \bigl\|e^{it(-\partial_x^2)^{\frac a2}}f(x)\bigr\|_{L^{2}_{t,x}(I\times \R)} \leq C_{I}\|f\|_{L^2(\R)}, 
    \end{equation*}
    one has
    \begin{equation*}
        \bigl\|e^{it(-\partial_x^2)^{\frac a2}}f(x)\bigr\|_{L^{4}_{t,x}(I\times \R)} \leq C_{I}\|f\|_{W^{\varepsilon, 4}(\R)}. 
    \end{equation*}
    Combining this and Lemma \ref{lem:nonwetght_to_weight}, we obtain the desired result. 
\end{proof}

Now, we are ready to prove Corollary \ref{c:local_smoothing_modulation_lu}. 
\begin{proof}[Proof of Corollary \ref{c:local_smoothing_modulation_lu}]
    By proceeding the same argument in the proof of Corollary \ref{c:local_smoothing_modulation}, we have
    \begin{equation*}
        \bigl\|e^{it(-\partial_x^2)^{\frac a2}}f(x)\bigr\|_{L^{4}_{t,x}(I \times \R)} \lesssim_{\varepsilon} \Bigl(\sum_{k\in\Z} \bigl\|e^{it(-\partial_x^2)^{\frac a2}}\sq{k}f(x)\bigr\|^{2}_{L^{4}_{t,x}(\eta_{I} \times \R)}\Bigr)^{\frac{1}{2}}
    \end{equation*}
    where $0\le \eta_{I}\in \mathcal{S}(\R^d)$ with $\eta_{I} \gtrsim 1$ on $I$. 
    We apply Proposition \ref{prop:weighted_local_smoothing} to each summand. Then it follows that 
    \begin{align*}
        \bigl\|e^{it(-\partial_x^2)^{\frac a2}}f(x)\bigr\|_{L^{4}_{t,x}(I \times \R)} &\lesssim_{\varepsilon} \Bigl(\sum_{k\in\Z} \bigl\|\sq{k}f(x)\bigr\|^{2}_{W^{\varepsilon, 4}(\R)}\Bigr)^{\frac{1}{2}}\\
        &\lesssim \|f\|_{M^{\varepsilon}_{4,2}(\R)}. 
    \end{align*}
\end{proof}

\subsection{Local well-posedness for the cubic NLS}\label{subsection wellposedness}
In this section, we present an application of our local smoothing estimate for the cubic fractional NLS: 
\begin{align}
    \begin{cases}
        i\partial_{t}u + (-\partial^2_x)^{\frac a2}u = \pm |u|^2 u\\
        u(\cdot, 0) = f(\cdot). 
    \end{cases}
\end{align}

As a corollary of our local smoothing estimate in modulation spaces, we obtain the following local well-posedness result. 
\begin{cor}
    Let $a > 2$, $\varepsilon > 0$, and $2\le p < \frac{8}{3}$. Then the cubic NLS is locally well-posed in $X\in \{ M^{\varepsilon}_{2,4}(\R),M^{\varepsilon}_{4,2}(\R), M^{\varepsilon}_{p,p}(\R),W^{\varepsilon, 4}(\R)\}$. 
\end{cor}

\begin{proof}
    By Corollary \ref{c:local_smoothing_modulation}, \ref{c:local_smoothing_modulation_lu}, the interpolation for modulation spaces, and Theorem \ref{thm:rogers_seeger}, we already have corresponding linear estimates
    \begin{equation*}
        \|e^{it(-\partial_x^2)^{\frac a2}}f(x)\|_{L^{4}_{t,x}([0,T]\times \R)} \leq C_{T}\|f\|_{X}. 
    \end{equation*}
    On the other hand, the stationary phase method, $TT^{*}$ argument, and the Christ-Kiselev lemma invoke that the nonlinear estimate 
    \begin{equation*}
        \Bigl\|\int^{t}_{0}e^{i((t- s)(-\partial^2_x)^{\frac a2}}F(s)ds\Bigr\|_{L^{4a}_{t}L^{4}_{x}([0,T]\times \R)}\lesssim \|F\|_{L^{\frac{4a}{4a-1}}_{t}L^{\frac{4}{3}}_{x}([0,T]\times \R)}
    \end{equation*}
    holds. Then by applying the standard iteration argument, we obtain the desired result. 
\end{proof}

\appendix

\section{van der Corput's lemma}

Here we recall a classical lemma for oscillatory integrals. For further details, see for instance \cite{bigstein}.

\begin{lem}[van der Corput's lemma]\label{l:van der Corputs}
Let $-\infty<a<b<\infty$. Suppose $\phi$ is a real-valued function in $(a,b)$ satisfying $|\phi^{(k)}(x)|\geq1$ for all $x\in (a,b)$. If $k=1$ and $\phi'$ is monotonic on $(a,b)$, or simply $k\geq2$, then there exists a constant $c_k>0$ (independent of $\phi$ and $\lambda$) such that 
\[
\Big|\int_a^b e^{i\lambda\phi(x)}\,\d x\Big|
\leq 
c_{k}\lambda^{-\frac1k}.
\]
\end{lem}

Using this, one can derive the following bound for level sets. For the further discussion, see \cite{carbery_christ_wright1999}. 

\begin{lem}[van der Corput's lemma in sublevel set form]\label{l:vandercoroput sublevel}
Let $k\geq 1$. For an interval $I\subset\R$ suppose that $u\in C^k(I,\R)$ satisfies $|u^{(k)}(t)|\geq 1$ for all $t\in I$. Then, there exists a constant $C_k>0$, depending only on $k$, such that for every $\alpha>0$
\[
|\{t\in I:|u(t)|\leq \alpha\}|
\leq 
C_k\alpha^{\frac1k}.
\]
\end{lem}

\begin{proof}
We first consider the case $k=1$. Since $u'$ is continuous and $|u'(t)|\geq 1$ for all $t\in I$, the derivative $u'$ has a single sign. Hence, $u$ is monotone. If $s,t$ belong to the set 
\[
E_\alpha:=\{t\in I:|u(t)|\leq \alpha\},
\]
then the mean value theorem implies 
\[
|t-s|\leq |u(t)-u(s)|\leq 2\alpha.
\]
Thus, $|E_\alpha|\leq 2\alpha$ holds.

Suppose next that $k\geq 2$. By the assumption and repeating Rolle's theorem,
$u'$ has at most $k-1$ zeros, and hence $u$ is monotone on at
most $k$ intervals. Consequently $E_\alpha$ is a union of at most
$k$ intervals, namely,
\[
E_\alpha\cap [-R,R] 
=
\bigcup_{j=1}^N I_j, \quad N\leq k
\]
for intervals $I_j$.

If we set $\lambda=(100\alpha)^{-1}$, we have $|\lambda u(t)|\leq 100^{-1}$ on $E_\alpha$. Hence, Lemma~\ref{l:van der Corputs} gives
\[
\Big|\int_{E_\alpha\cap[-R,R]}e^{i\lambda u(t)}\,\d t\Big|
\leq 
\sum_{j=1}^N\Big|\int_{I_j}e^{i\lambda u(t)}\,\d t\Big|
\lesssim_k
\alpha^{\frac1k}.
\]
On the other hand, 
\[
\Big|\int_{E_\alpha\cap[-R,R]}e^{i\lambda u(t)}\,\d t\Big|
\geq 
\int_{E_\alpha\cap[-R,R]}\cos(\lambda u(t))\,\d t
\gtrsim
|E_\alpha|.
\]

\end{proof}

\section{Alternative proof of Theorem~\ref{t:main} when \texorpdfstring{$b=a$}{b = a}}

In this appendix, we present an alternative proof for our main theorem when $b = a$ without using van der Corput's lemma. For simplicity, we only consider the following special case of Theorem~\ref{t:main}.

\begin{thm}
    Let $a \in (0,\infty)\backslash\{1\}$. Let $\delta \in (0,1)$ and $F\in\mathcal{S}(\R^2)$ be a Schwartz function with $\supp \widehat{F}\in \Gamma_{a}(\delta)$. 
    \begin{itemize}
        \item For $a \ge 2$, 
        \[
        \|F\|_{L^{4}(\R^2)} \lesssim \Bigl\|\Bigl(\sum_{k\in\Z}|F_{k}|^{2}\Bigr)^{\frac 12}\Bigr\|_{L^4 (\R^2)}
        \]
        holds where $\widehat{F}_{k}(\xi) := \widehat{F}(\xi)1_{[k-\delta^{1/a}, k+\delta^{1/a}]}(\xi_1)$. 

        \item For $a\in (0,2)\backslash\{1\}$, \[
        \|F\|_{L^{4}(\R^2)} \lesssim \Bigl\|\Bigl(\sum_{k\in\Z}|F_{k}|^{2}\Bigr)^{\frac 12}\Bigr\|_{L^4 (\R^2)}
        \]
        holds where $\widehat{F}_{k}(\xi) := \widehat{F}(\xi)1_{[k-\delta^{1/2}, k+\delta^{1/2}]}(\xi_1)$
    \end{itemize}
\end{thm}

\begin{proof}[Sketch of the proof]
    By $L^4$ expansion as in the proof of Theorem \ref{t:main}, it holds that 
    \[
    \int_{\R^2}\Bigl|\sum_{k}F_{k}\Bigr|^{4}dx
        = \int_{\R^2}\Bigl(\sum_{k_1, k_2}\wdh{F_{k_1}}*\wdh{F_{k_2}}\Bigr)\overline{\Bigl(\sum_{k_3, k_4}\wdh{F_{k_3}}*\wdh{F_{k_4}}\Bigr)}d\xi. 
    \]
    To show our theorem, we need to count the number of quadruples $(k_1,k_2,k_3,k_4)$ for which $\supp \wdh{F_{k_1}}*\wdh{F_{k_2}}$ and $\supp\overline{\wdh{F_{k_3}}*\wdh{F_{k_4}}}$ overlap with each other. For this sake, it is enough to consider the system 
    \begin{align*}
        \begin{cases}
            \tilde\xi_1 + \tilde\xi_2 = \tilde\xi_3 + \tilde\xi_4\\
            \tilde\xi^a_1 + \tilde\xi^a_2 = \tilde\xi^a_3 + \tilde\xi^a_4 + O(\delta)
        \end{cases}
    \end{align*}
    where $(\tilde\xi_i , \tilde\xi^a_i )\in\supp \widehat{F}_{i}$ for $i = 1,2,3,4$. By commuting the convolution product and taking the complex conjugate, we may assume 
    \[
    \tilde\xi_1\ge\tilde\xi_3 \ge\tilde\xi_4\ge \tilde\xi_2. 
    \] 
    \medskip
    
    \noindent\textbf{Case $a\in [2,3]$.} 
    
    Our task is to deduce $|\tilde\xi_1 - \tilde\xi_3| \lesssim \delta^{\frac{1}{a}}$ or $|\tilde\xi_1 - \tilde\xi_4| \lesssim \delta^{\frac{1}{a}}$. So, we also may assume that $\tilde{\xi}_1 + \tilde{\xi}_2 \gtrsim \delta^{\frac{1}{a}}$ otherwise $\tilde\xi_i$ are coming from blocks that are $\delta\times \delta^{\frac{1}{a}}$-close to the origin (the number of such blocks is $O(1)$).
    Set $M := \tilde{\xi}_1 + \tilde{\xi}_2$ and set a function $f(\,\cdot\,;a)$ as
    \begin{equation*}
        f(t;a) = t^a + (M - t)^a
    \end{equation*}
    for $t \in [0, M]$. Then the derivatives of $f(\,\cdot\,;a)$ are
    \begin{align*}
        \partial_t f(t;a) = a(t^{a-1} - (M - t)^{a-1}),\\
        \partial_t^2 f(t;a) = a (a-1) (t^{a-2} + (M - t)^{a-2}). 
    \end{align*}
    Note that $\partial_t f(\frac{M}{2};a) = 0$ and 
    $\partial_t^2 f(t;a) \geq a(a -1)M^{a-2}$ for all $t\in [0,M]$
    because $0 < a-2 < 1$. Then 
    \begin{align*}
        &\tilde{\xi}^a_1 + \tilde{\xi}^a_2 - \tilde{\xi}^a_3 - \tilde{\xi}^a_4\\
        &= \int^{\tilde{\xi}_1}_{\tilde{\xi}_3} \partial_t f(t;a)\,dt\\
        &= \int^{\frac{\tilde{\xi}_1 - \tilde{\xi}_3}{2}}_{-\frac{\tilde{\xi}_1 - \tilde{\xi}_3}{2}}
        \partial_t f\Bigl(u + \frac{M}{2} +\frac{\tilde{\xi}_1 - \tilde{\xi}_4}{2};a \Bigr)\,du .
    \end{align*}
    Here in the second equality, we changed the variables as
    $t \to u + \frac{\tilde{\xi}_1 + \tilde{\xi}_3}{2}$ and used the relation
    \[
    \frac{\tilde{\xi}_1 + \tilde{\xi}_3}{2} - \frac{M}{2}
    = \frac{\tilde{\xi}_1 + \tilde{\xi}_3}{2} - \frac{\tilde{\xi}_3+ \tilde{\xi}_4}{2}
    = \frac{\tilde{\xi}_1 - \tilde{\xi}_4}{2}.
    \]
    We write
    \[
    h(x;a) := \partial_t f\Bigl(x + \frac{M}{2};a\Bigr),
    \qquad
    x\in \Bigl[-\frac{M}{2}, \frac{M}{2}\Bigr].
    \]
    Obviously, $h(\,\cdot\,;a)$ is odd and $h(0;a) = 0$. Hence, the above equality becomes
    \begin{align*}
        &= \int^{\frac{\tilde{\xi}_1 - \tilde{\xi}_3}{2}}_{-\frac{\tilde{\xi}_1 - \tilde{\xi}_3}{2}}
        \Bigl[
        h\Bigl(u + \frac{\tilde{\xi}_1 - \tilde{\xi}_4}{2};a\Bigr)
        - h(u;a)
        \Bigr]\,du .
    \end{align*}
    We apply the mean value theorem to the integrand and get
    \begin{equation*}
        h\Bigl(u + \frac{\tilde{\xi}_1 - \tilde{\xi}_4}{2};a\Bigr)
        - h(u;a)
        = \frac{\tilde{\xi}_1 - \tilde{\xi}_4}{2}\,\partial_x h(\theta;a)
    \end{equation*}
    for some $\theta\in \Bigl(u,u + \frac{\tilde{\xi}_1 - \tilde{\xi}_4}{2}\Bigr)$.
    On the other hand, we have the pointwise bound
    \begin{align*}
        \partial_x h(t;a)
        &= \partial_t^2 f\Bigl(t + \frac{M}{2};a\Bigr)\\
        &\ge a (a-1)M^{a-2}
    \end{align*}
    for all $t\in \Bigl[-\frac{M}{2}, \frac{M}{2}\Bigr]$. Therefore, it follows that 
    \begin{equation*}
        \tilde{\xi}^a_1 + \tilde{\xi}^a_2 - \tilde{\xi}^a_3 - \tilde{\xi}^a_4
        \ge \frac{a(a-1)}{2}M^{a-2}
        (\tilde{\xi}_1 -\tilde{\xi}_3)(\tilde{\xi}_1 - \tilde{\xi}_4).
    \end{equation*}
    Combining this and the second condition
    $|\tilde{\xi}^a_1 + \tilde{\xi}^a_2 - \tilde{\xi}^a_3 - \tilde{\xi}^a_4|
    \lesssim \delta$,
    we obtain
    \begin{equation*}
        M^{a-2}
        |\tilde{\xi}_1 - \tilde{\xi}_3|
        |\tilde{\xi}_1 - \tilde{\xi}_4|
        \lesssim \delta .
    \end{equation*}
    Note that we assumed
    $M = \tilde{\xi}_1 + \tilde{\xi}_2 \gtrsim \delta^{\frac{1}{a}}$.
    Then one has
    \begin{equation*}
        |\tilde{\xi}_1 - \tilde{\xi}_3|
        |\tilde{\xi}_1 -\tilde{\xi}_4|
        \lesssim \delta \cdot \delta^{\frac{2}{a} - 1}
        = \delta^{\frac{2}{a}} .
    \end{equation*}
    This immediately gives
    $|\tilde{\xi}_1 - \tilde{\xi}_3| \lesssim \delta^{\frac{1}{a}}$
    or
    $|\tilde{\xi}_1 - \tilde{\xi}_4| \lesssim \delta^{\frac{1}{a}}$.
    The implicit constants in these inequalities do not depend on $\delta$.

    \medskip
    \noindent\textbf{Case $a > 3$.}

    For this case, we use the inequality
    \[
    \tilde{\xi}^a_1 + \tilde{\xi}^a_2 - \tilde{\xi}^a_3 - \tilde{\xi}^a_4
    \geq
    M\frac{a}{a-2}
    \bigl(
    \tilde{\xi}^{a-1}_1 + \tilde{\xi}^{a-1}_2
    - \tilde{\xi}^{a-1}_3 - \tilde{\xi}^{a-1}_4
    \bigr)
    \]
    and hence reduce the problem to the previous case. This inequality can be shown rewriting the left-hand side by the integration of $f(t;a)$ and considering an inequality \[
    \partial_{t}f(t;a) \geq \frac{M}{2}\frac{a}{a-2} \partial_{t} f(t;a-1) \quad \left(t\in \left[\frac{M}{2},M\right]\right).
    \]

    \medskip
    \noindent\textbf{Case $a\in (0,2)\backslash \{1\}$. }

    As in the previous cases, we may assume $\tilde\xi_1 + \tilde\xi_2 \gtrsim \delta^{\frac{1}{2}}$. 
    For this case, we consider the inequality
    \[
    \delta \gtrsim \tilde{\xi}^a_1 + \tilde{\xi}^a_2 - \tilde{\xi}^a_3 - \tilde{\xi}^a_4 = \int^{\tilde\xi_{1}}_{\tilde\xi_{3}}\partial_{t}f \left(t ;a \right)dt
    \]
    and apply the Jensen inequality to this. Then we get $|\tilde\xi_1 - \tilde\xi_3|\lesssim \delta^{\frac 12}$ and we obtain the desired estimate. 
\end{proof}

\section*{Acknowledgments}
The first and the third author are supported in part by the Croatian Science Foundation under the project IP-2022-10-5116 (FANAP).
The second author would like to thank Prof. Mitsuru Sugimoto for helpful discussions and valuable comments. He also would like to thank Prof. Neal Bez for pointing out the classical proof for the bilinear Strichartz estimate. He was supported by Grant-in-Aid for JSPS Fellows No. 24KJ1228. 

\section*{Declarations}

\subsection*{Data availability}
This work has no associated data. 

\subsection*{Conflict of interest}
On behalf of all authors, the corresponding author states that there is no conflict of interest.

\bibliographystyle{abbrvurl}
\bibliography{ref_revised}

\end{document}